\documentclass[11pt]{article}
\usepackage{amsthm, amsmath, amssymb, amsfonts, url, booktabs, tikz, setspace, fancyhdr, bm}
\usepackage{hyperref}
\usepackage{amsthm}
\usepackage{geometry}
\usepackage{hyperref, enumerate}
\usepackage[shortlabels]{enumitem}
\usepackage[babel]{microtype}
\usepackage[english]{babel}
\usepackage[capitalise]{cleveref}
\usepackage{comment}
\usepackage{bbm}
\usepackage{csquotes}
\usepackage{mathabx}
\usepackage{tikz}
\usepackage{subcaption}
\usepackage{graphicx}
\usepackage{float}
\usepackage{xcolor}
\usetikzlibrary{positioning, arrows.meta, shapes.geometric}

\counterwithin{figure}{section}

\newtheorem{theorem}{Theorem}[section]
\newtheorem{prop}[theorem]{Proposition}
\newtheorem{lemma}[theorem]{Lemma}

\newtheorem{conj}{Conjecture}
\newtheorem{claim}[theorem]{Claim}

\newtheorem{fact}[theorem]{Fact}
\newtheorem{obs}[theorem]{Observation}

\theoremstyle{definition}
\newtheorem{defn}[theorem]{Definition}
\newtheorem*{defn-non}{Definition}

\definecolor{rosepink}{RGB}{255,102,204}
\definecolor{dateplum}{HTML}{993366}
\definecolor{darkdateplum}{RGB}{128,0,32}
\definecolor{lightdateplum}{RGB}{219,112,147}
\definecolor{darkred}{RGB}{139,0,0}
\definecolor{lightred}{RGB}{240,130,100}

\newlist{Case}{enumerate}{2}
\setlist[Case, 1]{%
    label           =   {\bfseries Case \arabic*.},
    labelindent=1em ,labelwidth=1.3cm, labelsep*=1em, leftmargin =!
}
\setlist[Case, 2]{%
    label           =   {\bfseries Subcase \arabic{Casei}.\arabic*.},
    labelindent=-1em ,labelwidth=1.3cm, labelsep*=1em, leftmargin =!
}
\newcommand{\Yemph}[1]{\textcolor{black}{\emph{#1}}}

\newenvironment{poc}{\begin{proof}[Proof of claim]}{\end{proof}}

\newcommand{\C}[1]{{\protect\mathcal{#1}}}

\newcommand{\ceil}[1]{\lceil #1\rceil}
\newcommand{\floor}[1]{\lfloor #1\rfloor}

\usepackage{todonotes}

\newcommand{\MIS}{\mathcal{I}_{\textup{max}}(G)}

\title{Sublinear hitting sets for some geometric graphs} 

\author{Xinbu Cheng\thanks{Laboratory of Mathematics and Complex Systems, Ministry of Education, School of Mathematical Sciences, Beijing Normal University, Beijing, China. Emails: chengxinbu2006@sina.com.}
\and
Xinqi Huang\thanks{School of Mathematical Sciences, University of Science and Technology of China, Hefei, Anhui 230026,
China.
Emails:
\{huangxq,rong\_ming\_yuan\}@mail.ustc.edu.cn
}
\and
Mingyuan Rong\footnotemark[2]
\and
Zixiang Xu\thanks{Corresponding author. Extremal Combinatorics and Probability Group (ECOPRO), Institute for Basic Science (IBS), Daejeon, South Korea. Emails: zixiangxu@ibs.re.kr. Supported by IBS-R029-C4.}
}

\begin{document}
\date{}
\maketitle
\begin{abstract}
For an $n$-vertex graph $G$, let $h(G)$ denote the smallest size of a subset of $V(G)$ such that it intersects every maximum independent set of $G$. A conjecture posed by Bollob\'{a}s, Erd\H{o}s and Tuza in early 90s remains widely open, asserting that for any $n$-vertex graph $G$, if the independence number $\alpha(G) =\Omega(n) $, then $h(G) = o(n)$. In this paper, we establish the validity of this conjecture for various classes of graphs, Our main contributions include:
\begin{enumerate}
  \item We provide a novel unified framework to find sub-linear hitting sets for graphs with certain locally sparse properties. Based on this framework, we can find hitting sets of size at most $O(\frac{n}{\log{n}})$ in any $n$-vertex even-hole-free graph (in particular, chordal graph) and in any $n$-vertex disk graph, with linear independence numbers. 
  \item Utilizing geometric observations and combinatorial arguments, we show that any $n$-vertex circle graph $G$ with linear independence number satisfies $h(G)\le O(\sqrt{n})$. Moreover, we extend this methodology to more general classes of graphs.
   \item We show the conjecture holds for those hereditary graphs having sublinear balanced separators.
\end{enumerate}
 We also show that $h(G)$ can be upper bounded by constants for several sporadic families of graphs with large independence numbers.

\medskip
\noindent {{\it Key words and phrases\/}: Bollob\'{a}s-Erd\H{o}s-Tuza conjecture, hitting set, geometric graphs}

\smallskip

\noindent {{\it AMS subject classifications\/}: 05C69}


\end{abstract}

\section{Introduction}
\subsection{Overview}
For a graph $G=(V,E)$, an \Yemph{independent set} is a set of vertices in $G$, no two of which are adjacent. A \Yemph{maximum independent set} is an independent set of the largest possible size for the graph $G$, and this size is called the \Yemph{independence number} of $G$ and is usually denoted by $\alpha(G)$. We say $T\subseteq V(G)$ is a \Yemph{hitting set} of the graph $G$ if $T\cap I\neq\emptyset$ holds for every maximum independent set $I\subseteq V(G)$. Let $h(G)$ be the smallest size of a hitting set $T\subseteq V(G)$. Assume that $n$ is sufficiently large whenever this is needed. Bollob\'{a}s, Erd\H{o}s and Tuza~(see,~\cite{1999BookErdos,1991ErdosCollection}) raised the following conjecture.

\begin{conj}[\cite{1999BookErdos,1991ErdosCollection}]\label{conj:BETConj}
    Let $G$ be an $n$-vertex graph with $\alpha(G)=\Omega(n)$, then $h(G)=o(n)$.
\end{conj}

For this conjecture, there was a clever observation of Hajnal~\cite{1965Hajnal}, (also see Lemma~1 in~\cite{2011JGTLandon}) which states that for all maximum independent sets in $G$, say $A_{1},A_{2},\ldots,A_{s}$, the following inequality always holds:
\begin{equation*}
|A_{1}\cap A_{2}\cap\cdots\cap A_{s}|+|A_{1}\cup A_{2}\cup\cdots\cup A_{s}|\ge 2\alpha(G). 
\end{equation*}
As a simple consequence, one can see that if $\alpha(G)>\frac{|G|}{2}$, then $|A_{1}\cap A_{2}\cap\cdots\cap A_{s}|>0$, which implies that $h(G)=1$. However,~\cref{conj:BETConj} even remains open in the case of $c=\frac{1}{2}-\varepsilon$ for any small $\varepsilon>0$. A recent major progress of~\cref{conj:BETConj} due to Alon~\cite{2021AlonHittingSet} showed that for every $n$-vertex regular graph $G$ with $\alpha(G)>(\frac{1}{4}+\varepsilon)n$ for arbitrarily small $\varepsilon>0$, we have $h(G)\le O(\sqrt{n\log{n}})$. Alon's proof elegantly combined the container method~\cite{2016AlonSpencer,2015JAMSBalogh,2015HypergraphContainer} and the above observation of Hajnal. On the other hand, Alon~\cite{2021AlonHittingSet} provided an explicit construction of graph $G$ with $\alpha(G)=\Theta(n)$ and $h(G)=\Theta(\sqrt{n})$, which means that the $o(n)$ term in~\cref{conj:BETConj} cannot be replaced by $o(\sqrt{n})$. Moreover, Alon provided another type of construction of graph $G$ with $n=2^{m}$ vertices such that $\alpha(G)=\sum_{i=0}^{\frac{m}{2}-t}\binom{n}{i}$ and $h(G)=\Theta(m^{2})$. This construction is mainly based on some ideas from discrepancy theory and Kleitman's isodiametric inequality~\cite{1966Kleitman}, which also disproves a conjecture of Dong and Wu~\cite{2022SIAMWuDong}. In particular, Alon~\cite{2022Alon} suggested to study whether 3-colorable graphs satisfy~\cref{conj:BETConj}, and showed that for every integer $q\ge 1$, there exists an $8$-colorable graph $G_{q}$ with $h(G)>q$, see~\cite[Theorem~1.3]{2023P5Free}.

Notice that the condition ``maximum independent sets'' cannot be replaced by ``maximal independent sets''. To see this, let $G$ be a book graph that consists of a clique $K_{(1-c)n}$, and the other $cn$ vertices are only adjacent to all vertices in the clique $K_{(1-c)n}$. Then we need exactly $(1-c)n+1$ vertices to hit all maximal independent sets.

One way to better understand~\cref{conj:BETConj} is to test whether it holds for specific classes of graphs. In recent work of Hajebi, Li and Spirkl~\cite{2023P5Free}, the authors put this idea into practice. More precisely, they proved that for any induced $P_{5}$-free graph $G$, $h(G)\le f(\omega(G))$, where $P_{5}$ is a path with $5$ vertices, $f: \mathbb{N}\rightarrow \mathbb{N}$ is a function and $\omega(G)$ is the largest size of a clique in $G$.
A family of graphs $\mathcal{G}$ is termed \Yemph{hereditary} if, for any graph $G\in \mathcal{G}$ and any induced subgraph $G'$ of $G$, we have $G'\in \mathcal{G}$. We say that a hereditary class $\mathcal{G}$ of graphs has the Erd\H{o}s-Hajnal property if there exists some $\varepsilon>0$ such that every graph $G\in\mathcal{G}$ satisfies either $\alpha(G)>|V(G)|^{\varepsilon}$ or $\omega(G)>|V(G)|^{\varepsilon}$. The famous Erd\H{o}s-Hajnal conjecture~\cite{1989DAMEH} states that every hereditary class of graphs has Erd\H{o}s-Hajnal property. Hajebi, Li and Spirkl~\cite{2023P5Free} recently provided a new approach to this conjecture. More precisely, they showed that for a hereditary class $\mathcal{G}$ of graphs, if every graph $G\in\mathcal{G}$ satisfies $h(G)\le \textup{Poly}(\omega(G))$, then $\mathcal{G}$ satisfies Erd\H{o}s-Hajnal property.

However, this approach does not always work. Note that the family of graphs with VC-dimension at most $d\in\mathbb{N}$ is hereditary.
\begin{theorem}\label{thm:NotHitBounded}
   There exists some graph $G$ with VC-dimension at most $3$ and $\omega(G)=3$ such that $h(G)\ge \frac{\sqrt{|G|}}{2}$. 
\end{theorem}
~\cref{thm:NotHitBounded} follows from Alon's construction in~\cite[Theorem~1.3]{2021AlonHittingSet}, we will discuss this in~\cref{sec:BVCBVC}.

\subsection{Our contributions}
In this paper, we will demonstrate the validity of~\cref{conj:BETConj} for several important classes of graphs, some of which stem from geometry, while others exhibit highly favorable combinatorial properties. 

A \Yemph{hole} in a graph is an induced subgraph which is a cycle of length at least four. A hole is called even if it has an even number of vertices. An \Yemph{even-hole-free graph} is a graph with no even holes. One of the most important sub-classes of the even-hole-free graphs is the family of chordal graphs. A \Yemph{chordal graph} is a simple graph in which every cycle of length four and greater has a cycle chord.

The absence of even-hole-free graphs attracts broad attention in the field of graph theory. The renowned strong perfect graph theorem, established by Chudnovsky, Robertson, Seymour, and Thomas~\cite{2006AnnalsStrongPerfect}, states that a graph achieves perfection precisely when it forbids both odd holes and odd anti-holes (defined as the complements of holes). Graphs without any odd holes and odd anti-holes are termed Berge graphs. It has been observed that the structure of even-hole-free graphs bears striking similarities to that of Berge graphs, for instance, the prohibition of induced $4$-holes necessitates the exclusion of any anti-hole with a length of at least $6$. For a positive integer $\gamma$, we define a vertex $v\in V(G)$ to be \emph{$\gamma$-simplicial} if its neighborhood $N_{G}(v)$ forms a disjoint union of at most $\gamma$ cliques. Chudnovsky and Seymour~\cite{2023JCTBEvenHoleFree2w} recently provided a valid proof and showed that every even-hole-free graph contains a $2$-simplicial vertex, thereby settling a conjecture posed by Reed~\cite{2001Reed}. It then follows that $h(G)\le 2\omega(G)+1$ by~\cref{obs:Basic}. We also refer interested readers to a survey for even-hole-free graphs~\cite{2010DMEVenHoleSurvey}. Here we confirm that~\cref{conj:BETConj} holds for any even-hole-free graph.

\begin{theorem}\label{thm:evenHoleFree}
    Let $G$ be an even-hole-free graph with $\alpha(G)=cn$ for some $c>0$, then we have $h(G)\le \frac{5n}{\log{n}}$.
\end{theorem}

To prove~\cref{thm:evenHoleFree}, we develop a framework to show that graphs with certain locally sparse properties satisfy~\cref{conj:BETConj}, see~\cref{thm:General:Sparse}. Under this framework, we can also prove~\cref{conj:BETConj} holds for the \Yemph{disk graph}, which is the intersection graph of disks in two-dimensional space.

\begin{theorem}\label{thm:Disk}
 Let $G$ be an $n$-vertex disk graph with $\alpha(G)=cn$ for some $c>0$, then we have $h(G)\le\frac{25n}{\log{n}}$.
\end{theorem}
 Our subsequent result concerning disk graphs follows in a similar vein as the recent work of Chudnovsky and Seymour~\cite{2023JCTBEvenHoleFree2w}, which is also related to the problem of Hajebi, Li and Spirkl~\cite{2023P5Free}.
\begin{theorem}\label{thm:11-simDisk}
   Every disk graph $G$ contains an $11$-simplicial vertex. Consequently, every disk graph $G$ satisfies $h(G)\le 11\omega(G)+1$.
\end{theorem}

The constant $11$ above might be improved and we make no attempt to optimize it here. There are some special sub-families of the disk graphs, for instance, a \Yemph{unit disk graph} is a disk graph with an intersection graph consisting of disks of unit size, one can easily show $h(G)=O(1)$ when $\alpha(G)=cn$ since unit disk graph $G$ is induced $K_{1,6}$-free. A \Yemph{coin graph} is a disk graph with the additional condition that no two disks have more than one point in common. Note that the circle packing theorem, first proven by Paul Koebe~\cite{1936CirclePacking} in 1936, states that a graph is planar only if it is a coin graph. Therefore any $n$-vertex coin graph has at most $3n-6$ edges, which implies we can only use $6$ vertices to hit all maximum independent sets in any coin graph by~\cref{obs:Basic}.

Moreover, the results in~\cref{thm:Disk} can be generalized to the intersection graph of higher dimensional balls by simple modification of our proof. One of the useful auxiliary results is that, in any $d$-dimensional space, the number of unit balls that can touch a unit ball is bounded by $f(d)$ for some function $f$~\cite{1979HighDimSphere}. For instance, $f(2)=6$ implies that the unit disk graph is induced $K_{1,6}$-free.

Next, our attention shifts to the \Yemph{circle graph} $G$, where $V(G)$ comprises a finite set of chords on a circle, with two vertices being adjacent if and only if their corresponding chords intersect. Circle graphs and their representations play a pivotal role in numerous unexpected fields. Exemplary instances include knot theory~\cite{1993InvenKnotCircleGraph} and quantum computing~\cite{2005PRAquan}. Also, circle graphs have garnered considerable interest both in graph theory~\cite{1994JCTBCircleGraph,2020CircleGraphNaji,2009SangilCircleGraph} and computer science~\cite{2022CircleISMFCS,1985SIAMJCOmputing,2022CircleIsomorphism}. In particular, a very recent breakthrough of Davies and McCarty~\cite{2021BLMSChiCircle} showed that the chromatic number of circle graphs is quadratic in its clique number, which significantly improves the previous results in~\cite{1997DMCircleGraph,1988PashaCircle}. Shortly after this, the upper bound was further improved by Davies~\cite{2022PAMSCircle}, who showed the sharp upper bound $\chi(G)\le O(\omega(G)\log{\omega(G)})$ when $G$ is a circle graph. Further exploration of circle graphs can be found in an informative textbook~\cite{2004SurveyCircleGraph}. We can show that~\cref{conj:BETConj} also holds for circle graphs.

\begin{theorem}\label{thm:CircleGraphs}
    Let $G$ be a circle graph with $\alpha(G)=cn$ for some $c>0$, then $h(G)\le C\sqrt{n}$ for some constant $C$.
\end{theorem}
Indeed, the result can be extended from circle graphs to more general graphs in an abstract setting, as shown in~\cref{thm:2-incomparable K-intersecting}. This will be further discussed in~\cref{section:CircleBetterBound}.

Let $\mathcal{F}\subseteq 2^{[n]}$ be a set system, the \Yemph{Vapnik-Chervonenkis dimension} (for short, \Yemph{VC-dimension}) of $\mathcal{F}$ is the largest integer $d$ for which there exists a subset $S\subseteq [n]$ with $|S|=d$ such that for every subset $B\subseteq S$, one can find a member $A\in\mathcal{F}$ with $A\cap S=B$. We say that the graph $G$ has VC-dimension $d$ if the set system $\mathcal{F}:=\{N_{G}(v):v\in V(G)\}$ induced by the neighborhood of vertices in $G$ has VC-dimension $d$. The VC-dimension is one of the most useful combinatorial parameters that measure its complexity. The notion was introduced by Vapnik and Chervonenkis~\cite{1971TPAVCVC} in 1971, as a tool in mathematical statistics.

We can derive the sharp upper bound on $h(G)$ for graphs with VC-dimension one. 
\begin{theorem}\label{thm:VC1}
    Let $G$ be an $n$-vertex graph with VC-dimension $1$, then $h(G)\le\max\{\frac{n}{\alpha(G)},3\}$. In particular, if $\alpha(G)=cn$ for positive constant $c\neq\frac{2}{5}$, then $h(G)\le\floor{\frac{1}{c}}$; if $\alpha(G)=\frac{2n}{5}$, then $h(G)\le 3$. 
\end{theorem}
The upper bounds on~\cref{thm:VC1} are sharp for different region of $c$. The exact results reflect a rare phenomenon, where the performance of $h(G)$ is highly sensitive to the selection of $c$. Indeed, when $c\neq\frac{2}{5}$, if $\floor{\frac{1}{c}}$ is an integer, one can take the complete balanced $\frac{1}{c}$-partite graph, otherwise, one can take complete $\ceil{\frac{1}{c}}$-partite graph with parts $A_{1},A_{2},\ldots,A_{\ceil{\frac{1}{c}}}$ such that $|A_{1}|=(1-c\cdot\floor{\frac{1}{c}})n$ and $|A_{i}|=cn$ for any $i\ge 2$. It is easy to check that the complete multipartite graph has VC-dimension one. Particularly when $c=\frac{2}{5}$, we can see $G=C_{5}[\frac{n}{5}]$ satisfies $h(G)=3$ and the VC-dimension of $G$ is one. Note that when $\frac{1}{3}<c\le \frac{1}{2}$, only when $\frac{2}{5}$, $h(G)$ might be $3$, while in other cases $h(G)$ must be strictly less than $3$.

For convenience, we construct a partially ordered diagram illustrated in~\cref{figure:Main} to depict the principal findings of this paper and elucidate the interconnections among various graph classes, as well as some open problems.

\begin{figure}
\centering
\hspace*{-1.5cm}
\begin{tikzpicture}[node distance=0.9cm,
    every node/.style={rectangle, draw, align=center}, line width=0.8pt, scale=0.5]
    

\node (circle-trapezoid) {Circle $k$-gon graph: $O(\frac{n\log{\log{n}}}{\log{n}})$};
\node[line width=2pt,red] (interval filament) [above=of circle-trapezoid] {Interval filament graph:~\cref{conj:IntervalFl}};
\node (even) [left=of interval filament] {Even-hole-free graph: $O(\frac{n}{\log{n}})$};
\node[line width=2pt,blue] (String) [above=of interval filament] {String graph:~\cref{conj:StringGraph}};
    \node (chordal) [below=of even] {Chordal graph: $O(\frac{n}{\log{n}})$};
    \node (circle) [right=of interval filament] {Circle graph: $O(\sqrt{n})$};
    \node (trapezoid) [below=of chordal] {Trapezoid graph: $\frac{n}{\alpha(G)}$};
    \node (incomparability) [right=of trapezoid] {Incomparability graph: $\frac{n}{\alpha(G)}$};
    \node (circular-arc) [below right=of incomparability] {Circular-arc graph: $O(\frac{n}{\alpha(G)})$};
    \node (interval) [below right=of trapezoid] {Interval graph: $\frac{n}{\alpha(G)}$};
     \node (Perfect) [below=of circle] {Perfect graph: $\frac{n}{\alpha(G)}$};
    \node (disk) [above=of even] {Disk graph: $O(\frac{n}{\log{n}})$};
    \node[line width=2pt,orange] (bvc) [above=of String] {Graph with bounded VC-dimension:~\cref{conj:GraphWithBVC}};
    \node (unit disk) [above=of disk] {Unit disk graph: $O(\frac{n}{\alpha(G)})$};
    \node (vc 1) [above=of bvc] {Graph with VC-dimension one: $\frac{n}{\alpha(G)}$};

\draw[-Latex] (even) to (chordal);
\draw[-Latex] (interval filament) to (circle);
\draw[-Latex] (interval filament) to (chordal);
\draw[-Latex] (circle-trapezoid) to  (circle);
\draw[-Latex] (Perfect) to (incomparability);
\draw[-Latex] (chordal) to (trapezoid);
\draw[-Latex] (circle-trapezoid) to (trapezoid);
\draw[-Latex] (interval filament) to (incomparability);
\draw[-Latex] (incomparability) to (trapezoid);
\draw[-Latex] (trapezoid) to (interval);
\draw[-Latex] (circular-arc) to (interval);
\draw[-Latex] (circle-trapezoid) to (circular-arc);
\draw[-Latex] (String) to (disk);
\draw[-Latex] (String) to (interval filament);
\draw[-Latex] (bvc) to (disk);
\draw[-Latex] (disk) to (unit disk);
\draw[-Latex] (bvc) to (vc 1);
\draw[-Latex] (bvc) to (circle);

\end{tikzpicture}

\caption{Upper bounds on $h(G)$ for various families of graphs $G$ with $\alpha(G)=\Omega(n)$. $A\rightarrow B$: $B$ is a sub-family of $A$}\label{figure:Main}
\end{figure}
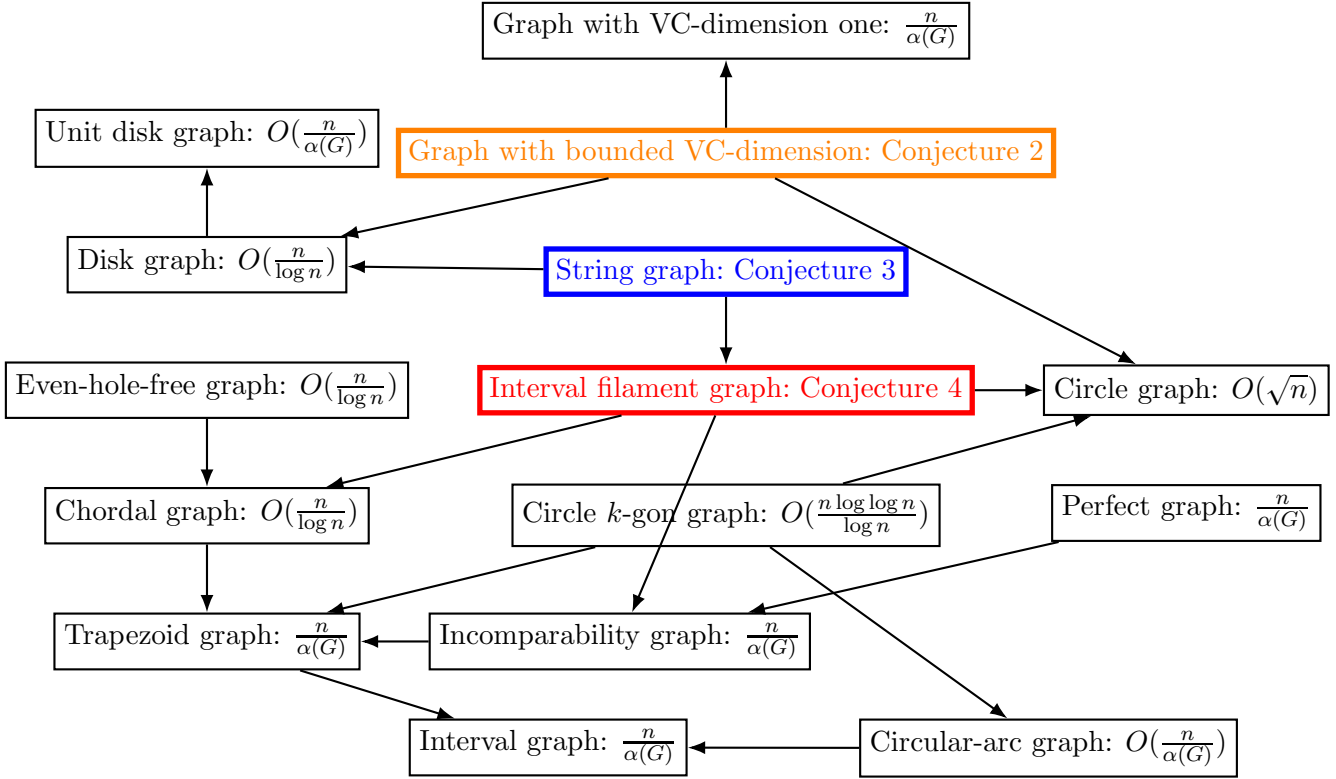

\medskip
{\bf \noindent Structure of this paper.} In~\cref{sec:Tools}, we introduce some basics on this topic and we provide the proof of~\cref{thm:VC1}. In~\cref{sec:LocalSparse}, we will show more general results for graphs with certain locally sparse properties in~\cref{thm:General:Sparse}. We also discuss the combinatorial properties of disk graphs and even-hole-free graphs and provide the proofs of~\cref{thm:evenHoleFree} and~\cref{thm:Disk} in detail. We will prove~\cref{thm:CircleGraphs} in~\cref{sec:Circle}. We conclude this paper and propose several conjectures related to the main topic of this paper in~\cref{sec:ConCludingRemarks}.

\medskip
{\bf \noindent Notations.}
For a vertex $u\in V(G)$, we will use $N_{G}(u)$ (or $N(u)$ if the subscript is clear) to denote the set of the neighborhood of $u$. Let $\delta(G):=\min_{v\in V(G)}|N_{G}(v)|$ be the minimum degree of $G$. Let $N_G(A)=\bigcup_{a\in A}N_G(a)$. For two sets $A$ and $B$, let $A\setminus B:=\{x\in A,x\notin B\}$ and let $A\triangle B:=(A\setminus B)\cup(B\setminus A) $ be the symmetric difference between $A$ and $B$. For a graph $H$, we write $H[t]$ the \Yemph{$t$-blowup} of $H$ by replacing every vertex of $H$ by an independent set of size $t$ and every edge of $H$ by a copy of $K_{t,t}$ and we write $H[\cdot]$ when mentioning a blow-up of $H$ without specifying its size (and the sizes of these independent sets could be different). We use $\MIS$ to denote the collection of all maximum independent sets in $G$. We say a graph $G$ is \Yemph{acyclic} if $G$ does not contain any cycle.

The notations $o(\cdot)$, $O(\cdot)$, $\Theta(\cdot)$ and $\Omega(\cdot)$ have their usual asymptotic meanings. We use $\log$ throughout to denote the base $2$ logarithm. For the sake of clarity of presentation, we omit floors and ceilings and treat large numbers as integers whenever this does not affect the argument.

\section{Warm-ups}\label{sec:Tools}
\subsection{Basics}
Concerning the main problem studied in this paper, there are some straightforward observations that were utilized in recent works~\cite{2022Alon,2023P5Free}. For completeness, we also include brief proofs.
\begin{prop}\label{obs:Basic}\
    \begin{enumerate}
        \item[\textup{(1)}] If $G$ has connected components $G_{1},G_{2},\ldots,G_{m}$, then $h(G)\le\min\limits_{i}h(G_{i})$.
        \item[\textup{(2)}] For any vertex $v\in V(G)$, $\{v\}\cup N_{G}(v)$ is a hitting set of $G$. In particular, $h(G)\le\delta(G)+1$.
        \item[\textup{(3)}] If $G$ is a bipartite graph, then $h(G)\le 2$.
    \end{enumerate}
\end{prop}
\begin{proof}[Proof of~\cref{obs:Basic}]
For (1), note that every $I\in\MIS$ is a disjoint union of $I_{i}\in\mathcal{I}_{\textup{max}}(G_{i})$ for $i\in [m]$, therefore, to hit all maximum independent sets in $G$, it suffices to hit all maximum independent sets in a connected component. For (2), if there is some $I\in\MIS$ such that $I\cap N_{G}(v)=\emptyset$, then $v\in I$. For (3), let $G=A\cup B$, if $|A|\neq |B|$, then $\alpha(G)>\frac{|G|}{2}$, which yields $h(G)=1$ by Hajnal's observation. Assume $|A|=|B|$, if there is a subset $S\subseteq A$ such that $|S|>|N(S)|$, then $\alpha(G)\ge |S\cup B\setminus N(S)|>\frac{|G|}{2}$, which implies $h(G)=1$. Otherwise, by Hall's theorem~\cite{1935Hall}, there is a perfect matching between $A$ and $B$. Take one edge $ab$ from this matching and put two endpoints $a$ and $b$ into the hitting set. It is easy to check that there is no independent set of size $\frac{|G|}{2}$ in the induced subgraph on subset $V(G)\setminus\{a,b\}$.
\end{proof}

\subsection{Graphs with VC-dimension $1$: Proof of~\cref{thm:VC1}}
Let $G$ be a graph with VC-dimension $1$ and $\alpha(G)=\alpha$. 
\begin{lemma}\label{lemma:VC1P5free}
    Let $G$ be a graph with VC-dimension one, then $G$ is induced $P_{5}$-free.
\end{lemma}
\begin{proof}[Proof of~\cref{lemma:VC1P5free}]
    Suppose there is an induced copy of $P_{5}:=x_{1}x_{2}x_{3}x_{4}x_{5}$, then $\{x_{2},x_{4}\}$ is shattered, a contradiction.
\end{proof}
Moreover, we can also obtain the following simple structural property.
\begin{lemma}\label{claim:triangleP5}
Let $G$ be a graph with VC-dimension one. If there is a triangle in $G$ with vertices $a_{1},a_{2},a_{3}$, then every vertex in $V(G)\setminus\{a_{1},a_{2},a_{3}\}$ is adjacent to at least two vertices in $\{a_{1},a_{2},a_{3}\}$.
\end{lemma}
\begin{proof}[Proof of~\cref{claim:triangleP5}]
    If a vertex $v\in V(G)\setminus\{a_{1},a_{2},a_{3}\}$ is not adjacent to any vertex in $\{a_{1},a_{2},a_{3}\}$, then the set $\{a_1, a_2\}$ is shattered. Moreover, if $v$ is adjacent to exactly one of $\{a_{1},a_{2},a_{3}\}$, by symmetry, assume $va_{1}\in E(G)$, then $\{a_{2},a_{3}\}$ is shattered.
\end{proof}
 We will also take advantage of the structural property of triangle-free graphs without induced $P_{5}$, which was proven by Sumner~\cite{1981P5Trainglefree} in 1981.
\begin{theorem}[\cite{1981P5Trainglefree}]\label{thm:P5K3free}
    Let $G$ be an induced $P_{5}$-free and triangle-free graph, then each connected component of $G$ is either a bipartite graph, or a copy of $C_{5}[\cdot]$.
\end{theorem}
Suppose that $G$ is triangle-free, then by~\cref{thm:P5K3free}, if $G$ is bipartite, then~\cref{obs:Basic} yields that $h(G)\le 2$, otherwise $G=C_{5}[\cdot]$ with parts $F_{1},F_{2},F_{3},F_{4},F_{5}$, and one can see $\alpha(G)\ge\frac{2n}{5}$. We further notice that if $c\neq\frac{2}{5}$, then $h(G)\le 2$. To see this, if $c\neq\frac{2}{5}$, then there exists some pair $(i,j)$ with $|i-j|\not\equiv 1\pmod{5}$ such that $|F_{i}\cup F_{j}|<\frac{2n}{5}$. Without loss of generality, assume $|F_{1}\cup F_{3}|<\frac{2n}{5}$, then we can pick arbitrary one vertex $u\in F_{4}$ and $v\in F_{5}$ respectively. Obviously for every $I\in\MIS$, $I\cap\{u,v\}\neq\emptyset$, which yields $h(G)\le 2$. when $c=\frac{2}{5}$, if $|F_{i}|=\frac{n}{5}$ for each $i\in [5]$, then $h(G)=3$, otherwise, we can see $h(G)\le 2$ via the identical argument as above. 

Next we can assume that there is a triangle with vertex set $\{a_{1},a_{2},a_{3}\}$. Take an arbitrary $I:=\{v_{1},v_{2},\ldots,v_{\alpha}\}\in\MIS$. Suppose that $I\cap\{a_{1},a_{2},a_{3}\}=\emptyset$, then by~\cref{claim:triangleP5}, $v_{1}$ must be adjacent to at least two vertices in $\{a_{1},a_{2},a_{3}\}$. Without loss of generality, we can assume that $v_{1}a_{1}\in E(G)$ and $v_{1}a_{2}\in E(G)$. As $v_{1},a_{1},a_{2}$ together form a copy of triangle, then for every vertex $v_{j}\in I\setminus\{v_{1}\}$, $v_{j}$ must be adjacent to both of $a_{1}$ and $a_{2}$ by~\cref{claim:triangleP5}, which indicates that all of vertices in $I$ are adjacent to both of $a_{1}$ and $a_{2}$. Moreover, as $a_{3}\notin I$, then $a_{3}$ must be adjacent to some vertex in $I$, say $a_{3}v_{k}\in E(G)$, which implies that $a_{1},a_{3},v_{k}$ form a copy of triangle. Again by~\cref{claim:triangleP5}, any vertex $v_{\ell}\in I\setminus\{v_{k}\}$ should be adjacent to $a_{3}$. Now let $v\in V(G)\setminus(I\cup\{a_{1},a_{2},a_{3}\})$. First, $v$ must be adjacent to at least two vertices in $\{a_{1},a_{2},a_{3}\}$, and we assume $a_{1}v,a_{2}v\in E(G)$ without loss of generality. As $v\notin I$, there must be some vertex $v_{q}$ such that $v_{q}v\in E(G)$, which implies that $a_{1}v_{q}v$ forms a copy of triangle, then by~\cref{claim:triangleP5}, every vertex in $I$ should be adjacent to the vertex $v$.

One can also apply the same argument to show that if $|I\cap\{a_{1},a_{2},a_{3}\}|=1$, every vertex $v\in V(G)\setminus I$ is adjacent to all vertices in $I$. We omit the repeated details here.

Therefore, if $G$ contains a copy of triangle as a subgraph, then all of the maximum independent sets of $G$ must be disjoint, which implies that $h(G)\le\frac{n}{\alpha}$ in this case. This finishes the proof.

\section{A general framework for locally sparse graphs}\label{sec:LocalSparse}

In this section, we provide a general framework to find a relatively small hitting set. This framework performs well for those graphs with linear independence number and certain locally sparse property. As examples, here we show that the even-hole-free graphs and disk graphs satisfy the property.

\subsection{A general framework}

\begin{theorem}\label{thm:General:Sparse}
    Let $c>0$, $s\in\mathbb{N}$ and $G$ be an $n$-vertex graph with $\alpha(G)=cn$. If for any $I_{1},I_{2}\in\MIS$,  $e(G[I_{1}\triangle I_{2}])\le s\cdot|I_{1}\triangle I_{2}|$, then $h(G)\le \frac{5sn}{\log{n}}$.
\end{theorem}
\begin{proof}[Proof of~\cref{thm:General:Sparse}]
Let $\delta=\frac{a}{\log{n}}$ where $a=5s$ is a constant and let $b=0.98a$. Select a maximal sequence $I_1,\ldots,I_t\in\MIS$ such that for each $k\in[t]$, $|I_k\setminus \big(\bigcup_{i=1}^{k-1}I_{i}\big)|>\frac{\delta n}{bk}$. Let $W=\bigcup_{i=1}^{t}I_{i}$ and $G'=G[W]$. Then 
$$n\ge |W|\ge cn+\sum\limits_{i=2}^{t}\frac{\delta n}{bi}=cn+(\log{t}-O(1))\frac{\delta n}{b},$$ 
and so $t\le e^{b/\delta+O(1)}<\frac{n^{0.99}}{2s}$. Then, by the assumption, we have
\begin{equation}\label{claim:SparseBipartite}
    e_G(I_{1},\bigcup_{i=2}^{t}I_{i})\le \sum_{i=2}^{t}e(G[I_{1}\triangle I_{i}])\le 2s(t-1)cn<cn^{1.99}.
\end{equation}

Now suppose that $G'=G$. By~\cref{claim:SparseBipartite} and the pigeonhole principle, there exists a vertex $v\in I_{1}$ with $d_{G}(v)<n^{0.99}$ and we are done by~\cref{obs:Basic}.

It remains to consider the case that $G'\neq G$. Again by~\cref{claim:SparseBipartite}, we can pick a subset $C\subseteq I_{1}$ with $|C|=\frac{\delta n}{4s(t-1)}$ such that $e_G(C,\bigcup_{i=2}^{t}I_{i})\le 2s(t-1)|C|\le\frac{\delta n}{2}$. Note that 
$$ |C| + |N_{G'}(C)|\le |C|+ e_G(C,\bigcup_{i=2}^{t}I_{i})\le\delta n=\frac{5sn}{\log n}.$$ 
It suffices to show that $C\cup N_{G'}(C)$ hits all maximum independent sets in $G$. Fix an arbitrary $I\in\MIS$. The maximality of $t$ infers that  $|I\setminus W|\le \frac{\delta n}{b(t+1)}$. Suppose to the contrary that $I\cap (C\cup N_{G'}(C))=\emptyset$. Then, $(I\cap W)\cup C$ is an independent set of size $|C|+|I|-|I\setminus W|\ge\frac{\delta n}{4s(t-1)}+(cn-\frac{\delta n}{b(t+1)})>cn$ as $b=0.99a>4s$, contradicting to $\alpha(G)=cn$. 
\end{proof}
\subsection{Even-hole-free graphs: Proof of~\cref{thm:evenHoleFree}}
 A well-known result states that the number of edges in an acyclic graph is less than the number of vertices. Based on~\cref{thm:General:Sparse}, to prove~\cref{thm:evenHoleFree}, it suffices to show the following lemma.

  \begin{lemma}\label{claim:NoCycle}
      Let $G$ be an even-hole-free graph. Then, for any $I_{1},I_{2}\in \MIS$, $G[I_{1}\triangle I_{2}]$ is acyclic.
  \end{lemma}

\begin{proof}[Proof of~\cref{claim:NoCycle}]
    Consider a minimal cycle $C$ in $G'=G[I_{1}\triangle I_{2}]$ if exists. By minimality, $C$ is an induced cycle in $G$. As $G'$ is bipartite, $C$ is an induced even cycle, a contradiction. Thus no such cycle exists and $G'$ is acyclic as desired.  
\end{proof} 

\subsection{Disk graphs: Proof of~\cref{thm:Disk}}
Similarly as above, based on~\cref{thm:General:Sparse}, to prove~\cref{thm:Disk}, it suffices to show the following lemma.

\begin{lemma}\label{lemma:SparseDisk}
      Let $G$ be an $n$-vertex disk graph with $\alpha(G)=cn$ for some $c>0$. For any $I_{1},I_{2}\in\MIS$, $e(G[I_{1}\triangle I_{2}])\le 5|I_{1}\triangle I_{2}|$.
  \end{lemma}
\begin{proof}[Proof of~\cref{lemma:SparseDisk}]
Firstly, we prove the following claim using the geometric intuition that a disk cannot intersect simultaneously with too many non-intersecting larger disks. For a vertex $v$, we denote $\rho(v)$ the radius of the disk corresponding to $v$.
\begin{claim}\label{claim:SmallestCircle}
    Let $v\in V(G)$ be a vertex with radius $\rho(v)$, then there is no independent set $I$ of size $6$ in the induced subgraph $G[N_{G}(v)]$ such that every vertex $u$ in $I$ satisfies $\rho(u)\ge\rho(v)$.
\end{claim}
\begin{poc}
    Suppose there are seven disks $A,B_{1},B_{2},\ldots,B_{6}$ with centers $\boldsymbol{x},\boldsymbol{y}_{1},\boldsymbol{y}_{2},\ldots,\boldsymbol{y}_{6}$ and radii $r_{A},r_{1},r_{2},\ldots,r_{6}$ such that for all $i\in[6]$, $r_A\le r_i$, $A$ intersects $B_{i}$ and $B_{i}$ are pairwise disjoint. Then, for any $j\neq i$, $|\boldsymbol{x}-\boldsymbol{y}_{i}|\le r_{A}+r_{i}\le r_{j}+r_{i}<|\boldsymbol{y}_{i}-\boldsymbol{y}_{j}|$ and  $|\boldsymbol{x}-\boldsymbol{y}_{j}|\le r_{A}+r_{j}\le r_{i}+r_{j}<|\boldsymbol{y}_{i}-\boldsymbol{y}_{j}|$. Therefore, the angle $\angle\boldsymbol{y}_{i}\boldsymbol{x}\boldsymbol{y}_{j}$ is greater than $\frac{\pi}{3}$ for any pair of $1\le i<j\le 6$. This yields $2\pi<\angle\boldsymbol{y}_{1}\boldsymbol{x}\boldsymbol{y}_{2}+\cdots+\angle\boldsymbol{y}_{5}\boldsymbol{x}\boldsymbol{y}_{6}+\angle\boldsymbol{y}_{6}\boldsymbol{x}\boldsymbol{y}_{1}=2\pi$, a contradiction.
    \end{poc}
For a subset $S\subseteq V(G)$, let $d_{S}^{\ge}(u)$ be the number of neighbors $w\in N_{S}(u)$ such that $\rho(w)\ge\rho(u)$. For any $I_1,I_2\in\MIS$, By~\cref{claim:SmallestCircle} we have 
\begin{equation*}
e(G[I_{1}\triangle I_{2}])\le \sum\limits_{u\in I_{1}}d^{\ge}_{I_{2}\setminus I_{1}}(u)+\sum\limits_{v\in I_{2}}d^{\ge}_{I_{1}\setminus I_{2}}(v)\le 5|I_{1}\triangle I_{2}|.    
\end{equation*} 
\end{proof}

\subsection{Proof of~\cref{thm:11-simDisk}}
We prove that disk graphs have an $11$-simplicial vertex, i.e.~the set of its neighbors is the union of $11$ cliques.

For a disk graph $G$ and $v\in V(G)$, denote by $D_G(v)$ the disk corresponding to $v$ and $\rho_G(v)$ the radius of $D_G(v)$. Let $v_{0}\in V(G)$ be a vertex with minimum radius $\rho_G(v_{0})$ among all vertices in $G$. By rescaling, we may assume $\rho(v_{0})=1$. We shall prove that $v_0$ is the desired $11$-simplicial vertex.

    For each neighbor $v\in N_{G}(v)$, define a new disk $D'(v)$ with radius 1 internally tangent to $D_G(v)$ such that the common point $A=D'(v)\cap D_G(v)$ lies on the line going through centers of $D_G(v_0)$ and $D_G(v)$ (see~\cref{fig:Auxiliary Disk}). As each new disk $D'(v)$ is contained in the original $D_G(v)$, for any $v,w\in N_G(v_0)$, if $D'(v)\cap D'(w)\neq\varnothing$, then $vw\in E(G)$. As depicted in~\cref{fig:L10vertices}, the family of all new disks $\C D=\{D'(v): v\in N_G(v_0)\}$, all of which intersect $D_G(v_0)$, can be pierced by at most 11 points, implying that $v_0$ is an 11-simplicial vertex in $G$.

\begin{figure}[ht]
\begin{minipage}[t]{0.48\textwidth}
\centering
\begin{tikzpicture}
    \draw[black, line width=1.5pt] (-1,0) circle (1.5cm) node [ below] {$D_{0}$};
    
    \draw[red, line width=1.5pt] (2.5,0) circle (2.5cm) ;
    
    \draw[blue, line width=1.5pt] (1.5,0) circle (1.5cm) node [ below] {$D'$};

\draw[dashed] (-1,0) -- (2.5,0);

\node at (-0.3,0.3) {$A$};

\node[red] at (5.5,0) {$D$};

\end{tikzpicture}
\caption{Create the disk $D'$ according to the disks $D_{0}$ and $D$}\label{fig:Auxiliary Disk}
\end{minipage}
\begin{minipage}[t]{0.48\textwidth}
\centering
\begin{tikzpicture}
    \draw[line width=2pt] (0,0) circle (2cm);
    
    \filldraw [blue] (0,0) circle (2pt);
    
    \foreach \angle in {0,36,...,324}
    {
        \filldraw [blue] (\angle:3cm) circle (2pt);
    }
\end{tikzpicture}
\caption{Any unit disk $D$ which intersects $D_{0}$ must contain one of the eleven points}\label{fig:L10vertices}
\end{minipage}
\end{figure}

\section{Circle graphs: Proof of~\cref{thm:CircleGraphs}}\label{sec:Circle}
\subsection{Warm-up: Simple geometric proof of weaker upper bound}
In this part, we first show that circle graphs satisfy~\cref{conj:BETConj}, with weaker quantitative bound. This proof mainly relies on simple geometric observations and combinatorial argument.
\begin{theorem}\label{thm:WeakCircle}
    Let $G$ be an $n$-vertex circle graph with $\alpha(G)=cn$ for $c>0$, then $h(G)\le \frac{4n\log\log{n}}{\log{n}}$.
\end{theorem}

\begin{proof}[Proof of~\cref{thm:WeakCircle}]
Consider a circle $O$, where the vertex set of graph $G$ consists of $n$ distinct chords in $O$. By perturbation, we may assume that none of these $n$ chords are diameters of $O$. For each $v\in V(G)$, let $C_v$ denote its corresponding chord and $A_v$ the minor arc associated to $C_v$.

For two non-crossing chords $C_a,C_b$, we say that $C_a$ \Yemph{covers} $C_b$ if $A_a\supseteq A_b$. Furthermore, we say that $C_a$ \Yemph{directly covers} $C_b$ if there does not exist any other chord $C_c$ with $A_a\supseteq A_c\supseteq A_b$ (see~\cref{fig:DefCover}). For $X\in\MIS$ and $x\in X$, let $K_{X}(x)$ be the set of vertices in $X$ which are covered by $x$ and let $k_{X}(x):=|K_{X}(x)|$. Obviously, $0\le k_{X}(x)\le cn-1$. Note also that for any $a,b\in X$, if $C_a$ covers $C_b$, then $k_{X}(a)>k_{X}(b)$.
Now, for each $I\in\MIS$ and $v \in I$, assign it the color $k_{I}(v)$.

\begin{figure}[htbp]
\centering
\begin{minipage}[t]{0.45\textwidth}
\centering
\begin{tikzpicture}
    \draw[line width=1.5pt] (0,0) circle (2cm);
    
    \draw[blue, line width=1pt] (-0.5,-1.9) -- (-0.5,1.9) node[midway, right] {A}; 
    
    \draw[green, line width=1pt] (-0.9,-1.75) -- (-0.9,1.75)node[midway, left] {B}; 
    
    \draw[red, line width=1pt] (-2,0) -- (-1.5,1.34)node[midway, right] {C}; 
    \draw[orange, line width=1pt] (-1.8,-0.8) -- (-1,-1.7)node[midway, right] {D}; 
\end{tikzpicture}

\caption{$A$ covers $B,C,D$, but $A$ does not directly cover $C,D$}\label{fig:DefCover}
\end{minipage}
\begin{minipage}[t]{0.45\textwidth}
\centering
\begin{tikzpicture}

  \draw[line width=1.5pt] (0,0) circle (2cm);
  
\draw[blue, line width=1pt] (-1,-1.7) -- (-0.5,1.9) ;

\draw[blue, line width=1pt] (2,0) -- (1.5,-1.34);

\draw[blue, line width=1pt] (1.8,0.8) -- (1,1.7);

\draw[red,dashed, line width=1pt] (-2,0) -- (1.8,-0.8);

\end{tikzpicture}
\caption{Any chord can cross at most two distinct chords in $Z_{I}$}\label{fig:IntersectTwo}
\end{minipage}
\end{figure}

\begin{claim}\label{claim:CircleGraphColoring}
    Each vertex in $V(G)$ receives at most one color.
\end{claim}
\begin{poc}
    Suppose that some vertex $u$ receives two distinct colors, say for $X,Y\in\MIS$, $k_{X}(u)=i< j=k_{Y}(u)$. Note that chords in $X\setminus K_{X}(u)$ are disjoint from those in $K_{Y}(u)$ as they lie in different sides of $C_u$. Therefore, $(X\setminus K_{X}(u))\cup K_{Y}(u)$ is also an independent set of size $cn-i+j>cn$, a contradiction.
\end{poc}

Next, for $I\in\MIS$, consider a vertex $v\in I$ with minimum $k_I(v)$. The minimality of $v$ implies that it does not cover any other vertex in $I$ and so $k_I(v)=0$. Thus, we have the following claim.
\begin{claim}\label{claim:ColorZero}
For each $I\in\MIS$, the set $Z_{I}=\{v\in I: k_{I}(v)=0\}\neq\emptyset$.
\end{claim}

Note that for any $I\in \MIS$ and any three pairwise non-crossing chords in $Z_{I}$, any other chord can cross at most two of them (see~\cref{fig:IntersectTwo}). Thus, the number of edges between $Z_{I}$ and $V(G)\setminus I$ is at most $2(1-c)n$ and there exists a vertex $v\in Z_{I}$ with $d_{G}(v)\le\frac{2(1-c)n}{|Z_{I}|}$. Then by~\cref{obs:Basic}, $h(G)\le \frac{2(1-c)n}{|Z_{I}|}+1$. Consequently, if there is $I$ with $|Z_{I}|> \frac{\log{n}}{\log\log{n}}$, then we are done. We may then assume that $|Z_{I}|\le M:=\frac{\log{n}}{\log\log{n}}$ for all $I\in\MIS$.

Let $R\subseteq V(G)$ denote the set of vertices that do not belong to any maximum independent set. By~\cref{claim:CircleGraphColoring}, every vertex in $V(G)\setminus R$ is assigned a unique color in $\{0,1,\ldots, cn-1\}$. Partition $V(G)\setminus R$ by colors $V(G)\setminus R=V_{0}\cup V_{1}\cup V_{2}\cup\cdots\cup V_{cn-1}$, where any vertex $v\in V_{i}$ satisfies $k_{I}(v)=i$ for any $I\in\MIS$. By~\cref{claim:ColorZero}, $V_0$ is a hitting set for $\MIS$. However, it is hard to bound the size $V_0$ directly. We will instead use multiple layers of $V_i$ instead.

Fix an $I\in \MIS$ with $|Z_{I}|\le M$.

\begin{claim}\label{claim:directlyCover}
    For any $v\in I$, $v$ directly covers at most $M$ other vertices in $I$.
\end{claim}
\begin{poc}
    Let $U\subseteq I$ be the set of vertices in $I$ that are directly covered by $v$. Then no vertex in $U$ covers another. This, together with the fact that $U\subseteq I$ is independent, infers that the minor arcs $A_u$, $u\in U$, are pairwise disjoint. Then, considering a vertex $u'$ with minimal color with $A_{u'}\subseteq A_u$ for each $u\in U$, we see that $|Z_{I}| \geq |U|$ finishing the proof.
\end{poc}

\begin{claim}\label{claim:LevelOne}
  $I\cap \bigcup_{i=1}^{M}V_{i}\neq\varnothing$ and $|I\cap V_i|\le M$ for all $i\in[cn-1]$. In particular, $h(G)\le |\bigcup_{i=1}^{M}V_{i}|$.
\end{claim}
\begin{poc}
  Consider a vertex $u \in I \setminus Z_{I}$ with minimum color. By the minimality, $u$ can cover only vertices in $Z_I$ and so $k_I(u)\le |Z_I|$.

  Fix $i\in [cn-1]$, then the minor arcs $A_u$, $u\in I\cap V_i$, are pairwise disjoint. Again, considering a vertex $u'$ with minimal color with $A_{u'}\subseteq A_u$ for each $u\in I\cap V_i$, we see that $|Z_{I}| \geq |I\cap V_i|$.
\end{poc}

Let $t$ be the largest positive integer such that $(2M)^{t+1}\le \lfloor \frac{cn}{M+1}\rfloor$. Then by~\cref{claim:LevelOne}, we see that $I$ contains vertices with color at least $(2M)^{t+1}+1$, that is, $I\setminus (\bigcup_{i=1}^{(2M)^{t+1}}V_i)\neq\varnothing$.

\begin{claim}\label{claim:HighLevel}
  For each $0\le j\le t$, $I\cap \bigcup_{k=(2M)^{j}+1}^{(2M)^{j+1}}V_{k}\neq\varnothing$. In particular, $h(G)\le |\bigcup_{k=(2M)^{j}+1}^{(2M)^{j+1}}V_{k}|$.
  \end{claim}
  
\begin{poc}
We prove it by induction on $0\le j\le t$. The base case $j=0$ follows from~\cref{claim:LevelOne}. For the inductive step, suppose the claim holds for all $0\le j'<j\le t$. Consider a vertex $w$ in $I\setminus (\bigcup_{i=1}^{(2M)^{j+1}}V_i)$ with minimum color. Then, as $w$ can directly cover at most $M$ vertices in $I$ by~\cref{claim:directlyCover}, there must exist a vertex $u$ covering at least $\frac{(2M)^{j+1}+1-M}{M}\ge(2M)^j+1$ vertices. On the other hand $u$ covers at most $(2M)^{j+1}$ vertices by the minimality of $w$. This finishes the proof.
\end{poc}

As each $\bigcup_{k=(2M)^{j}+1}^{(2M)^{j+1}}V_{k}$, $j=0,1,\ldots, t$, is a hitting set for $\MIS$, by averaging, there is a hitting set of size at most $\lfloor\frac{n}{t+1}\rfloor$. By the choice of $t$, $(2M)^{t+2}\ge  \frac{cn}{M+1}$ and so $t\ge \frac{\log\frac{cn}{M+1}}{\log(2M)}-2\ge\frac{\log n}{4\log\log n}$. Thus, $h(G)\le \frac{4n\log\log{n}}{\log{n}}$. The proof is finished.
\end{proof}

\subsection{Proof of~\cref{thm:CircleGraphs} and more general results}\label{section:CircleBetterBound}
We will actually prove a more general result, before this, we introduce some necessary definitions.

\begin{defn}\label{def:incomparable}
    For a graph $G$ and two partial order $\subseteq$ and $\prec$ over $V(G)$, we say $(G,\subseteq,\prec)$ is a $2$-incomparable graph if for any distinct vertices $a,b,c \in V(G)$, the followings hold:
    \begin{itemize}
        \item  $ab$ is an edge in $G$ if and only if $a$ and $b$ are incomparable in both $\subseteq$ and $\prec$.
        \item If $a \subseteq b$ then $a$ and $b$ are incomparable in $\prec$.
        \item If $a \prec b$ then $a$ and $b$ are incomparable in $\subseteq$.
        \item If $a\prec b$ and $c\subseteq a$ then $c\prec b$.
        \item If $a\prec b$ and $c\subseteq b$ then $a\prec c$.
    \end{itemize}
\end{defn}

By \Cref{def:incomparable}, we can directly obtain the following facts.

\begin{fact}\label{fact:dont share slaves}
    Suppose that $a$, $b$ and $c$ are distinct vertices in a $2$-incomparable graph $(G,\subseteq,\prec)$ satisfying $c\subseteq a$ and $c\subseteq b$,
    then $a$ and $b$ are incomparable in $\prec$.
\end{fact}

\begin{fact}\label{fact:independent-chain}
    Suppose that $I$ is an independent set of a $2$-incomparable graph $(G,\subseteq,\prec)$ such that every pair of elements in $I$ are incomparable in $\subseteq$ (or $\prec$), then $I$ is a chain in $\prec$ (or $\subseteq$).
\end{fact}

\begin{defn}\label{def:K-intersecting}
    Let $K: \mathbb{N \to \mathbb{N}}$ be a map.
    We say a $2$-incomparable graph $(G,\subseteq,\prec)$ is $K$-intersecting if for any chain $a_1\prec a_2\prec \dots \prec a_\ell$ and $x\in V(G)\backslash \{a_1,a_2,\dots,a_\ell\}$, the inequality $|N(x)\cap \{a_1,a_2,\dots,a_\ell\}|\leq K(|V(G)|)$.

    Especially, if the map $K$ is a constant map satisfying $K(x) = K_0$ for any $x\in \mathbb{N}$ and some integer $K_0$.
    We replace the notation $K$-intersecting by $K_0$-intersecting for simplicity.
\end{defn}
One of the equivalent definitions of a circle graph is the \emph{overlap graph}. A graph is an overlap graph if its vertices correspond one-to-one with intervals on a line, where two vertices are adjacent if and only if their intervals partially overlap — that is, the intervals have a non-empty intersection, but neither is fully contained within the other. Let $a \prec b$ indicate that the intervals of $a$ and $b$ are disjoint, with the interval of $a$ lying to the left of that of $b$. Similarly, let $a \subseteq b$ indicate that the interval of $a$ is contained within the interval of $b$, implying that $(a, b)$ is not an edge in $G$. It is then straightforward to verify that an overlap graph $G$, together with the relations $\subseteq$ and $\prec$, forms a $2$-intersecting $2$-incomparable graph.

Then our main result reads as follows, which also implies~\cref{thm:CircleGraphs}.

\begin{theorem}\label{thm:2-incomparable K-intersecting}
    For any fixed constants $0< \beta<1$,
    the following holds. Let $(G,\subseteq,\prec)$ be a $K$-intersecting $2$-incomparable graph with $|V(G)| = n$ and $K: \mathbb{N}\to \mathbb{N}$.
    If $\alpha(G) \ge \beta n$, then $h(G)\le 2\sqrt{\frac{1-\beta}{\beta}n K(n)}$.
    In particular, $h(G) = o(n)$ if $K(n) = o(n)$.
\end{theorem}

\begin{proof}[Proof of~\cref{thm:2-incomparable K-intersecting}]
    For any vertex $x\in V(G)$ and subset $A\subset V(G)$, let 
    \begin{equation*}
        D(x,\subseteq) := \{ y\in V(G): \ y\subseteq x \} \textup{ and }
        D_A(x,\subseteq) := A\cap D(x,\subseteq).
    \end{equation*}

    Assume that the conclusion does not hold.
    Without loss of generality, we can assume that each vertex $x\in V(G)$ is contained in at least one maximum independent set of $G$.

    Since for each $x\in V(G)$, we have $\{x\}\cup N(x)$ is a hitting set of $G$,
    so we directly have $\delta(G) \ge L : =  2\sqrt{\frac{1-\beta}{\beta}n K(n)}$. For any $I\in \MIS$, we define a local coloring map $C_I: I\to \mathbb{N}$ such that $C_I(x) = |D_I(x,\subseteq)|$.
   Furthermore, we define a global coloring map $C: V(G)\to \mathbb{N}$: for any $x\in V(G)$, let $C(x) = C_I(x)$ for some $I\in \MIS$ containing $x$.

    \begin{claim}\label{claim:local-global coloring}
        The coloring map $C$ is well-defined.
    \end{claim}
    \begin{poc}
        By the assumption above, for each $x\in V(G)$, there exists $I\in \MIS$ containing $x$. It then suffices to show that, if $x$ is contained in two different maximum independent sets $I$ and $J$ of $G$, $C_I(x) = C_J(x)$.  
        Suppose not, denote $C_I(x) = i< j =C_J(x)$.
        Let $I' = (I\backslash D_I(x,\subseteq)) \cup D_J(x,\subseteq)$.
        It is easy to check that $|I'| = |I\backslash D_I(x,\subseteq)| + |D_J(x,\subseteq)| = \alpha(G) -i +j > \alpha(G)$.
        It remains to show that $I'$ is also an independent set of $G$, which leads to a contradiction. To see this, if $I'$ is not an independent set, there exists $y\in D_J(x,\subseteq)$ and $z\in I\backslash D_I(x,\subseteq)$ such that $yz\in E(G)$.
        By the definition above, we have $y\subseteq x$, $zx\notin E(G)$.
        Since $z\notin D_I(x,\subseteq)$, we can see at least one of the events $z < x$, $x < z$ and $x\subseteq z$ occurs. Then one can easily check by \Cref{def:incomparable} that all these three possibilities imply $yz\notin E(G)$, which contradicts to the assumption.
        This completes the proof.
        
    \end{poc}

    \noindent 
    Now we denote $V_i = \{ x\in V(G): C(x) = i\}$ for any integer $0\le i\le n$.
    For any $I\in \MIS$, denote $M_I = V_0\cap I$.
    Observe that $M_I \neq \emptyset$ for any $I\in \MIS$, furthermore, we have the following claim.

    \begin{claim}\label{claim: small number of slaves}
        For any $I\in \MIS$,
        $0<|M_I|\le P$, where $P = \frac{(1-\beta)n K(n)}{L} = 2\sqrt{\beta (1-\beta) n K(n)}$.
    \end{claim}

    \begin{poc}
        Suppose there exists $I\in \MIS$ satisfying $|M_I|> P$.
        We will double count the cardinality of the following set
        \begin{align*}
            S_I = \{
            (x,v): x\notin I,\ v\in M_I,\ xv\in E(G)
            \}.
        \end{align*}
        On one hand,
        \begin{align}\label{ineq:S_I lower}
            |S_I|= \sum_{v\in M_I}\textup{deg}(v)
            \ge |M_I| \cdot\delta(G).
        \end{align}
        On the other hand,
        \begin{align}\label{ineq:S_I upper}
            |S_I| = \sum_{x\in V(G)\backslash I}|N(x)\cap M_I| 
            \le (n-\alpha(G)) K(n) \le (1-\beta)n K(n),
        \end{align}
        where the first inequality holds by \Cref{def:K-intersecting} and the observation that $M_I$ is a chain in $\prec$.

        By combining (\ref{ineq:S_I lower}) and (\ref{ineq:S_I upper}) we have
        \begin{align*}
            P< |M_I| \le \frac{(1-\beta)n K(n)}{\delta(G)}
            \le \frac{(1-\beta)n K(n)}{L} = P,
        \end{align*}
        which leads to a contradiction and completes the proof.
        \end{poc}

        Finally, we will construct many disjoint hitting sets of $G$ based on the coloring $C$ and apply pigeonhole principle to show that there exists a small one. More precisely, let $m = \frac{\beta n}{2P} = \frac{1}{4}\sqrt{\frac{\beta n }{(1-\beta )K(n)}}$. For each $t\in [\lfloor\frac{m}{2}\rfloor]$, denote $H_t = \bigcup_{i=0}^{L}V_{im + t}$.

        \begin{prop}\label{claim:find hitting set}
            For each $t\in [\lfloor\frac{m}{2}\rfloor]$, $H_t$ is a hitting set of $G$.
        \end{prop}

        Suppose that \Cref{claim:find hitting set} holds, we can get a hitting set of $G$ of size at most $2n/m =  2\sqrt{\frac{1-\beta}{\beta}n K(n)}$ by pigeonhole principle, which completes the proof.
        So it suffices to prove \Cref{claim:find hitting set}.

        \begin{proof}[Proof of \Cref{claim:find hitting set}]
            For each $t\in [m]$, it suffices to show that $I\cap H_t \neq \emptyset$ for any $I\in \MIS$. For each $x\in I$, denote $S_x = \{y\in V_0\cap I: y\subseteq x\}$ and $|S_x| = s_x$. We have the following claim.
            \begin{claim}\label{claim:no jump}
                For each $x\in I$ with $C(x)\ge \max\{(i-1)m + t +1, 0\}$ for some $i\in \{0,1,\ldots,L\}$, we have $s_x\ge i+1$.
                Conversely, if $s_x = i$, then $C(x)\le \max\{0,(i-1)m+t -1\}$.
            \end{claim}

        \begin{poc}
            We prove the claim by induction on $i$. When $i = 0$, it trivially holds that $s_x\ge 1$.

            Suppose the statement holds for all $j\le i$.
            Assume for contrary that there exists $x\in I$ with $C(x)\ge im +t+1$ and $2\le s_x\le i$.
            Denote $B$ to be a subset of $D_I(x,\subseteq)$ consisting of all elements $a$ such that $s_x = s_a$.
            Then it is easy to check that $B$ is a chain in $\subseteq$.
            Denote $A$ to be the set consisting of all the maximal elements of $D_I(x,\subseteq)\backslash B$ in $\subseteq$,
            then $A$ is a chain in $\prec$ by \Cref{fact:independent-chain}.
            Furthermore, we have the following observation by \Cref{fact:dont share slaves}.
            \begin{obs}\label{obs:dont share slaves}
                For two different elements $y$ and $z$ in $A$, $S_y \cap S_z = \emptyset$.
            \end{obs}

            By the definition of $A$ and the induction hypothesis, we know that each element $y\in A$ satisfies $s_y <i$ and then $C(y)\le (s_y-1) m+t-1$.
            By the definition of $B$ and \Cref{fact:independent-chain}, we know that $B$ is a chain in $\subseteq$ and the minimal element $b$ satisfies
            \begin{align*}
                C(b)= \sum_{y\in A} C(y) +1 
                &\le \sum_{y\in A}\big( (s_y-1)m +t-1 \big)+1
                = s_x m + |A|(t-m-1) +1\\ 
                &= (s_x-1)m+t-1 +(|A|-1)(t-m-1)+1 \le (s_x-1)m+t-1.
            \end{align*}
            Since $B$ is a chain in $\subseteq$, the color set $C(B) = \{C(z): z\in B\}|$ should be consecutive in $\mathbb{N}$.
            Moreover, since $(i-1)m+t\in [C(b), C(x)]$, there exists $z\in B\subset I$ such that $C(z) = (i-1)m+t$, which is a contradiction to the assumption that $I\cap H_t = \emptyset$.
        \end{poc}

        Now we will continue to prove \Cref{claim:find hitting set}.
        Let $Q$ be a subset of $I$ consisting of all the maximal elements in the partial order $\subseteq$.
        Similar to \Cref{obs:dont share slaves}, we have the similar result based on \Cref{fact:dont share slaves}.
        \begin{obs}\label{obs:dont have slaves again}
            For any two different elements $y$ and $z$ in $Q$, we have $S_y\cap S_z = \emptyset$.
        \end{obs}
        Therefore we have $\alpha(G) = |I|\le \sum_{y\in Q}C(y)$. By \Cref{claim:no jump} and \cref{obs:dont have slaves again} we further have
        \begin{align*}
            \sum_{y\in Q}C(y) \le \sum_{y\in Q}\max\{0,(s_y-1)m+t-1\}
            \le \sum_{y\in Q} s_y m
            \le
            |V_0\cap I|m \le Pm 
            = \frac{1}{2}\alpha(G)
            < \alpha(G),
        \end{align*}
        which leads to a contradiction and completes the proof of~\cref{claim:find hitting set}.
        \end{proof}  

        This finishes the proof.
    \end{proof}

\section{Concluding remarks}\label{sec:ConCludingRemarks}
In conclusion, our study delves into the longstanding open problem proposed by Bollob\'{a}s, Erd\H{o}s and Tuza, which posits that for any $n$-vertex graph with $\alpha(G)=\Omega(n)$, it suffices to employ $o(n)$ vertices to hit all maximum independent sets. Our investigation draws inspiration from the remarkable result, along with its elegant proof, by Alon~\cite{2021AlonHittingSet}, demonstrating validity of~\cref{conj:BETConj} for regular graphs with large independence number, as well as the work of Hajebi, Li, and Spirkl~\cite{2023P5Free}, establishing the existence of small hitting sets for induced $P_{5}$-free graphs.

Our primary contribution lies not only in confirming the validity of~\cref{conj:BETConj} across several significant graph classes but also in synthesizing concepts from diverse fields to either construct or demonstrate the existence of small hitting sets. We envision our ideas serving as a springboard for further exploration aimed at tackling this conjecture.

It is intriguing to note that the conjecture remains open even when $\alpha(G)=(\frac{1}{2}-\varepsilon)|G|$ for any small $\varepsilon>0$. We believe that addressing this weaker conjecture presents its own set of challenges worthy of pursuit.

\subsection{Constant-sized hitting sets in various sporadic graphs}
In addition to the graph classes we have discussed in~\cref{sec:Tools}, in many other classes of geometric graphs, we can also prove the existence of hitting sets of constant size. We select some representative ones to introduce the sketches briefly.
For a configuration $S$ in high dimensional space $\mathbb{R}^{d}$, let $\mu(S)$ be the Lebesgue measure of $S$, and $|\boldsymbol{x}-\boldsymbol{y}|$ represent the Euclidean distance between points $\boldsymbol{x},\boldsymbol{y}\in\mathbb{R}^{d}$, we define the diameter of $S$ as $\textup{Diam}(S):=\sup\{|\boldsymbol{a}-\boldsymbol{b}|:\boldsymbol{a},\boldsymbol{b}\in S\}$. Observe that in an intersection graph of configurations with positive measures and bounded diameter $D$, there is no large induced star, because the ball of radius $2D$ can only contain bounded number of pairwise disjoint configurations. Therefore, the following result follows.
\begin{prop}\label{thm:GeneralConvexSets}
  Let $\varepsilon>0$ be a real number and $d\ge 1$ be a positive integer. Let $\mathcal{S}:=\{S_{1},S_{2},\ldots,S_{n}\}\subseteq \mathbb{R}^{d}$ be a collection of $n$ sets such that for each $i\in [n]$, $\mu(S_{i})>\varepsilon$ and $\textup{Diam}(S_{i})\le D$ for some constant $D>0$. Let $G$ be the intersection graph of the sets in $\mathcal{S}$, then $h(G)= O_{\varepsilon,D,d}(\frac{n}{\alpha(G)})$. In particular, if $\alpha(G)=cn$ for some $c>0$, then $h(G)=O_{\varepsilon,D,d,c}(1)$.
\end{prop}

In low dimensional space, the \Yemph{interval graph} is formed from a set of intervals on the real line, with a vertex for each interval and an edge between vertices whose intervals intersect. It was known that any interval graph is an incomparability graph, therefore, by~\cref{thm:Incomparability}, we can see $h(G)\le \frac{n}{\alpha(G)}$ when $G$ is an $n$-vertex interval graph.

An interesting generalization of interval graph is \Yemph{circular-arc graph}, which is the intersection graph of a set of arcs on the circle. It has one vertex for each arc in the set, and an edge between every pair of vertices corresponding to arcs that intersect. Observe that, for any given maximum independent set $I$, if a vertex $v\in V(G)\setminus I$ such that $|N_{G}(v)\cap I|\ge 4$, then $v$ cannot belong to any maximum independent set in $G$. Based on this, the following result follows. 

\begin{prop}\label{thm:circular-arc}
Let $G$ be a circular-arc graph, then we have $h(G)\le \frac{3n}{\alpha(G)}-2$, and in particular if $\alpha(G)=cn$ for some $c>0$, we have $h(G)\le\frac{3}{c}-2$.
\end{prop}

\subsection{Geometric extensions of circle graphs}
We can also extend the technique in the proof of~\cref{thm:CircleGraphs} to a more general class of graphs. More precisely, a \Yemph{circle 
$k$-gon} is the region that lies between $k$
 or fewer non-crossing chords of a circle, no chord connecting the arcs between two other chords. The sides of a circle $k$-gon are either chords or arcs of the circle. A graph is a \Yemph{circle k-gon graph} if it is the intersection graph of circle $k$-gons in a common circle. One can see that any circle graph is indeed a circle $1$-gon graph. Using the almost identical argument as that in proof of~\cref{thm:WeakCircle}, one can show the following more general result.

\begin{theorem}\label{thm:CircleTrapezoid}
   Let $k\ge 1$ be a positive integer and let $G$ be a circle $k$-gon graph with $\alpha(G)=\Omega(n)$, then we have $h(G)\le O(\frac{n\log\log{n}}{\log{n}})$.
\end{theorem}

\subsection{The bridge between the balanced separators and the hitting sets}
For an $n$-vertex graph $G$, a subset $S\subseteq V(G)$ is a \Yemph{balanced separator} if every component of $G\setminus X$ has at most $\frac{2n}{3}$ vertices. Equivalently, $S\subseteq V(G)$ is a balanced separator if there is a partition $V:=S\cup V_{1}\cup V_{2}$ with $|V_{1}|,|V_{2}|\le\frac{2|V(G)|}{3}$ such that no vertex in $V_{1}$ is adjacent to any vertex in $V_{2}$.

Here, we establish a straightforward connection between the balanced separator and the hitting set. Let $\mathcal{G}_{\varepsilon}$ be a hereditary family of graphs such that for every graph $G\in \mathcal{G}$, there exists a balanced separator of size at most $\varepsilon|V(G)|$. Note that for any $n$-vertex graph, we can repeatedly take balanced separators and consider the subgraphs induced by smaller components by~\cref{obs:Basic}. Note that the union of those separators forms a hitting set, which is of size at most $\sum\limits_{i}\frac{\varepsilon n}{2^{i-1}}\le 2\varepsilon n$.

\begin{prop}\label{thm:SmallSeparatorSmallHitting}
  Let $G\in\mathcal{G}_{\varepsilon}$ be an $n$-vertex graph, then $h(G)\le 2\varepsilon n$.
\end{prop}
Therefore, for those hereditary graphs having sublinear balanced separators, they also have sublinear hitting sets. For instance, a classical result of Lipton and Tarjan~\cite{1979SIAMJAMPlanar} shows that every $n$-vertex planar graph contains a balanced separator of size $O(\sqrt{n})$. This result on planar graphs has also been generalized to graphs embedded in a surface of bounded genus~\cite{1984SeparatorGenus}, graphs with a forbidden minor~\cite{1990JAMSAlonSeymourThomas}, and intersection graphs of some geometric objects~\cite{2008AdvFoxPach,1997JACM}.

In addition to the aforementioned classical results, a recent study by Dvo\v{r}\'{a}k, McCarty, and Norin~\cite{2021SIDMASublinearSeparators} provided a sufficient condition for an intersection graph of compact convex sets in $\mathbb{R}^{d}$ to possess a balanced separator of sublinear size. Furthermore, Dvo\v{r}\'{a}k and Wood~\cite{2023SeparatorProduct} have elucidated the connection between balanced separators and other parameters such as tree-width, path-width, and tree-depth. For additional hereditary families of graphs with sublinear balanced separators, we recommend interested readers to refer to~\cite{2021SIDMASublinearSeparators} and the references therein.

\subsection{Some open problems}
\subsubsection{Graphs with bounded dimension}\label{sec:BVCBVC}
\cref{thm:VC1} indicates that any graph with VC-dimension one has a very small hitting set. While for any larger positive integer $d\ge 2$, it remains an interesting open problem.

\begin{conj}\label{conj:GraphWithBVC}
Let $d\ge 2$ be a positive integer, and let $G$ be an $n$-vertex graph with VC-dimension at most $d$. If $\alpha(G)=\Omega(n)$, then $h(G)=o(n)$.
\end{conj}

We anticipate that insights from several papers~\cite{2024HitInducedMathcing,2019DCGBVCEH,2021ErdosSchur,2023BVCSunflower,2023BVCErdosHajnal,2023SukMatchingLemma} may prove instrumental in our exploration. Motivated by the recent work of Nguyen, Scott and Seymour~\cite{2023BVCErdosHajnal}, who proved that graphs with bounded VC-dimension have Erd\H{o}s-Hajnal property, one might expect that $h(G)\le \textup{Poly}(\omega(G))$ holds for the graphs with bounded VC-dimension. However, this is not true. Alon's construction in~\cite[Theorem~1.3]{2021AlonHittingSet} showed that there is a graph $G$ with bounded VC-dimension and small clique number, but $h(G)$ cannot be upper bounded by any function of $\omega(G)$.

\begin{proof}[Proof of~\cref{thm:NotHitBounded}]
Let $k$ be a positive integer and $K: = \{1,2,\ldots,2k\}$. The set of vertices of $G$ is the set of all ordered pairs $(i,j)$ with distinct $i,j\in K$. Note that the number of vertices is $n=2k(2k-1)$. A pair of vertices $(a,b)$ and $(c,d)$ are adjacent if $b=c$ or $a=d$. Alon~\cite{2021AlonHittingSet} has already proved that $h(G)\ge\frac{\sqrt{n}}{2}$. It then remains to show $\omega(G)=3$ and VC-dimension is at most $3$. First, it is easy to check that any triangle in $G$ must be of the form $(a,b),(b,c),(c,a)$, where $a,b,c\in K$ are distinct, moreover, it is obvious that these three vertices have no common neighbor. For the VC-dimension of $G$, suppose $S=\{(a_{i},b_{i}):i\in [4] \}$ is a shattered set. Let $(x,y)$ be the vertex such that $N_{G}((x,y))\supseteq N_{G}(S)$. Then for each $1\le i\le 4$, either $a_{i}=y$ or $b_{i}=x$. If for each $1\le i\le 4$, $a_{i}=y$ (or $b_{i}=x$ respectively), then there is no vertex $(z,w)$ such that $|N_{G}((z,w))\cap S|=3$. Therefore, without loss of generality, we can assume there exist three vertices in $S$ of the form $(a_{1},x),(a_{2},x),(y,b_{3})$ (or $(b_{1},y),(b_{2},y),(a_{3},x)$), then it is easy to check that these three vertices have no common neighbor other than $(x,y)$. This finishes the proof. 
\end{proof}

In the theory of online learning, the \emph{Littlestone dimension} (for short, \emph{LS-dimension}) of $\mathcal{F}$ is another key parameter that measures complexity proposed by Littlestone~\cite{1988littlestone}. following the definition in~\cite{2023BVCSunflower}, it is defined as the largest integer $d$ such that there exists a full binary tree $T_{d}$ with $2^d$ leaves at the last level, where the leaves of $T_d$ are labeled by sets in $\mathcal{F}$ and all other vertices by elements of $[n]$. The tree $T_d$ is shattered by $\mathcal{F}$ if for every root-to-leaf path with labels $v_0, v_1, \ldots, v_{d-1}, F$, we have $v_i \in F$ if and only if the $(i+1)$-vertex along the path is the left-child of $v_i$, for all $0 \le i < d$. This parameter is particularly useful in online learning and mistake bounds, providing a combinatorial measure of the complexity of a set system and receiving attentions from different areas~\cite{2021Stoc,2022JACM,2023ITCSC}. By the definitions, we can see that for any given set system $\mathcal{F}$, the LS-dimension of $\mathcal{F}$ is no less than the VC-dimension. We say that the graph $G$ has LS-dimension $d$ if the set system $\mathcal{F} := \{N_{G}(v) : v \in V(G)\}$ induced by the neighborhood of vertices in $G$ has LS-dimension $d$. It would be interesting to establish that graphs with bounded LS-dimension satisfy~\cref{conj:BETConj}. At present, we can only confirm this for graphs with LS-dimension at most $2$, using a detailed structural analysis. We omit the full details here, which is similar as the proof of~\cref{thm:VC1}.

\begin{theorem}\label{thm:LS2}
    Let $G$ be an $n$-vertex graph with $\alpha(G)=\Omega(n)$ and LS-dimension at most $2$, then $h(G)= o(n)$.
\end{theorem}

\subsubsection{String graphs}
In a fundamental sense, curves represent the most versatile geometric objects on the plane, defined as the images of continuous functions $\phi:[0,1]\rightarrow\mathbb{R}^{2}$. A string graph emerges as the intersection graph of such curves. Notice that the majority of graph families examined within this paper align with the concept of string graphs. Consequently, the following conjecture assumes significant relevance.

\begin{conj}\label{conj:StringGraph}
Let $G$ be an $n$-vertex string graph.
\begin{enumerate}
    \item[\textup{(1)}] If $\alpha(G)=\Omega(n)$, then $h(G)=o(n)$.
    \item[\textup{(2)}] There exists some constant $s>0$ such that $h(G)\le \omega(G)^{s}$.
\end{enumerate}

\end{conj}
 The significant finding by Fox and Pach~\cite{2012AdvString} might be helpful, which asserts that any dense string graph must contain a dense incomparability graph as a spanning subgraph. \cref{conj:StringGraph}(2) is also motivated by a recent breakthrough by Tomon~\cite{2024JEMSStringEHConj}, who demonstrated that string graphs have the Erd\H{o}s-Hajnal property (also see a more flexible result~\cite{2023EUJC}). Building on the relationship discovered in~\cite{2023P5Free}, a proof of \cref{conj:StringGraph}(2) would offer a novel proof of Tomon's theorem and possibly lead to potential quantitative improvements.

We are also interested in the family of interval filament graphs introduced by Gavril~\cite{2000IntervalFilament}. A graph is an \Yemph{interval filament graph} if it is the intersection graph of continuous non-negative functions defined on closed intervals with zero-values on their endpoints. One of the motivations is that the family of interval filament graphs is a common generalization of the incomparability graphs, the circle graphs and the chordal graphs. We can also see that any interval filament graph $G$ with is a $2$-incomparable graph, but might be not $K$-intersecting for $K=o(n)$. 

\begin{conj}\label{conj:IntervalFl}
    Let $G$ be an $n$-vertex interval filament graph with $\alpha(G)=\Omega(n)$, then $h(G)=o(n)$.
\end{conj}

\subsubsection{Optimal construction}
We believe that~\cref{conj:BETConj} is true, but currently we are unable to prove it. Conversely, we conjecture that the $o(n)$ term in the conjecture is optimal, suggesting that the following conjecture may be valid.

\begin{conj}
For any positive constant $\varepsilon>0$, there exists an $n$-vertex graph $G$ with $\alpha(G)=cn$ for some constant $c>0$ and $h(G)\ge n^{1-\varepsilon}$.
\end{conj}

\section*{Acknowledgement}
Zixiang Xu would like to express sincere gratitude to Prof. Hong Liu for his extensive suggestions on the writing of this draft. The authors also thank Prof. Noga Alon for helpful comments.

\bibliographystyle{abbrv}
\bibliography{BETConj}

\end{document}